\documentclass[11pt]{article}
\usepackage[authoryear]{natbib}  

\usepackage{authblk}

\usepackage{color,tikz}
\usepackage{mathrsfs}
\usepackage{graphicx}
\usepackage{mathtools}
\usepackage{relsize}
\usepackage{exscale}
\usepackage{caption,subcaption}
\usepackage{amsmath}
\def\Item$#1${\item $\displaystyle#1$
   \hfill\refstepcounter{equation}(\theequation)}
\usepackage{amssymb}
\usepackage{textcomp}
\usepackage{euscript}

\usepackage{hyperref}
\usepackage{hyperref}
\usepackage{xcolor}
\usepackage{tcolorbox}
\usepackage[titletoc,title]{appendix}
\usepackage{upgreek}
\hypersetup{
    colorlinks,
    linkcolor={red!50!black},
    citecolor={blue!50!black},
    urlcolor={blue!80!black}
}
\makeatletter
\newcommand{\@giventhatstar}[2]{\left(#1\,\middle|\,#2\right)}
\newcommand{\@giventhatnostar}[3][]{#1(#2\,#1|\,#3#1)}
\newcommand{\giventhat}{\@ifstar\@giventhatstar\@giventhatnostar}
\makeatother

\usepackage{longtable}

\usepackage{float}
\usepackage[scaled]{helvet}
\usepackage[ruled]{algorithm2e}
\usepackage{algorithmicx}

\newcommand{\pp}{\mathbb{P}}

\newcommand{\ee}{\mathbb{E}}

\usepackage[utf8]{inputenc}
\usepackage{amsmath} 
\usepackage{amssymb}
\usepackage{amsthm}

\usepackage[margin=1in]{geometry}

\newcommand{\intg}{\mathbb{Z}}

\newcommand{\reals}{\mathbb{R}}

\newcommand{\mb}{\mathbf}

\newcommand{\lt}{\left}
\newcommand{\rt}{\right}
\newcommand{\wt}{\widetilde}

\newcommand{\mx}{\mb{x}}
\newcommand{\my}{\mb{y}}

\newcommand{\mD}{\mb{D}}
\newcommand{\mX}{\mb{X}}
\newcommand{\mY}{\mb{Y}}
\newcommand{\mZ}{\mb{Z}}

\newcommand{\mU}{\mb{U}}
\newcommand{\mL}{\mb{L}}
\newcommand{\scG}{\mathscr{G}}

\providecommand{\keywords}[1]
{
  \small	
  \textbf{\textit{Keywords---}} #1
}

\newtheorem{theorem}{Theorem}

\newtheorem*{example}{Example}

\newtheorem{lemma}{Lemma}
\newtheorem*{algorithm-non}{Algorithm}
\title{Perfect Sampling for Gibbs Point Processes Using \\Partial Rejection Sampling}
\author{Sarat B. Moka}
\author{Dirk P. Kroese}
\affil{\small School of Mathematics and Physics\\
The University of Queensland, Brisbane}

\date{}
\begin{document}

\maketitle

\begin{abstract}
We present a perfect sampling algorithm for Gibbs point processes, based on the partial rejection sampling of \citet{GJL17}. Our particular focus is on pairwise interaction processes, penetrable spheres mixture models and area-interaction processes, with a finite interaction range. For an interaction range $2r$ of the target process, the proposed algorithm can generate a perfect sample with $O(\log(1/r))$ expected running time complexity, provided that the intensity of the points is not too high.
\end{abstract}

\keywords{Perfect sampling, Partial-rejection sampling, Hard-core process, Strauss process, Pairwise interaction process, Area-interaction process, Penetrable spheres mixture model}

\section{Introduction}
        Various phenomena in physics, chemistry and biology are modelled by Gibbs point processes. A Gibbs point process --- or simply, Gibbs process --- is a  spatial point process whose distribution is absolutely continuous with respect to that of a Poisson point process (PPP). Pairwise interaction point (PIP) processes and  penetrable spheres mixture (PSM) models are two widely studied examples of Gibbs processes; see, for e.g., \cite{MW04, MH16, KM00, BN12, BT05}. The PIP family includes hard-core processes and Strauss processes. \\
        
        Perfect sampling for Gibbs processes is an active area of research. A sampling algorithm  for a given distribution is called {\em perfect} if it generates an exact sample from this distribution within a finite time. We refer to \citet{KW98, Fill98, KM00, GNL00, FFG02, MH12, MJM17, Gj18} for some of the existing perfect sampling algorithms for Gibbs processes. The methods in \citet{MJM17} and \citet{Gj18} generate perfect samples of hard-core processes. The other methods in the references above are applicable to more general Gibbs processes, including PIP processes and PSM models. Among these methods, the dominated coupling from the past (dCFTP) methods by \citet{KW98, KM00} and \citet{MH12} are shown to be efficient when the density of the points is small; see, for example, \citet{MH16}. As we show in this paper, for an interaction range $2r$ of the target Gibbs process and dimension $d$ of the points, the expected running time complexity of any dCFTP method is at least of order $\frac{1}{r^d} \log\lt(\frac{1}{r} \rt)$ even when the density of the reference PPP is very small. However, dCFTP algorithms are sequential and thus they do not take advantage of parallel computing. In this paper, we propose a method for generating perfect samples of PIP processes and PSM models using {\em partial rejection sampling} (PRS) method of \citet{GJL17} and show how one can obtain, using parallel computing, an expected running time complexity of $O(\log(1/r))$, provided that the density of the reference PPP is not too high. \\

The PRS method provides a general methodology for generating perfect samples from a product distribution, conditioned on none of a number of {\em bad events} occurring. Such problems are in general {\bf NP}-hard; see, for e.g., \citet{BGGGS16} and \citet{GJL17}. However, for certain types of parametric product distributions, the PRS algorithm is efficient and terminates within $O(\log n)$ iterations on average, where $n$ is the number of bad events. An additional feature of the PRS algorithm is that, unlike the dCFTP methods, it is distributive, in the sense that it allows parallel computation within each iteration. As a consequence, the PRS algorithm can  be implemented with $O(\log n)$ expected running time complexity. By exploiting the distributive property of the PRS, we use the PRS algorithm for generating perfect samples of Gibbs processes on a Euclidean subset $S$. In particular, a brief description of our contributions is as follows:
\begin{itemize}
\item We partition $S$ into a finite number of cells and define a product measure by ignoring the {\em cross} interactions between the cells. Further by defining appropriate bad events that depend on the cross interactions, we express the distribution of the target Gibbs process as the product distribution conditioned on none of the bad events occuring. This construction allows the generation of perfect samples using PRS.

\item To analyze the running time complexity of the proposed algorithm, we take $S = [0,1]^d$ and the intensity of the reference PPP as ${\kappa = \frac{\kappa_0}{\mathsf{v}_d r^d}}$ for some constant $\kappa_0$, where $2r$ is the interaction range of the Gibbs process and $\mathsf{v}_d $ is the volume of a $d$-dimensional sphere of unit radius. We consider the regime where $\kappa_0$ is fixed and $r$ goes to zero, and prove that if the volume of each cell is of order $r^d$, there exists a constant $\bar \kappa~>~0$ such that for all $\kappa_0 \leq \bar \kappa$, the expected running time complexity of the algorithm is $O\lt(\log \frac{1}{r}\rt)$ as a function of $r$. 
\item  To illustrate the application of the proposed algorithm, we consider a $d$-dimensional cubic grid partitioning of $S = [0,1]^d$ and conduct extensive simulations to estimate the expected number of iterations of the algorithm for different values of $\kappa_0$ and the interaction range $2r$. 
\end{itemize}
To the best of our knowledge, this is the first method for PIP processes and PSM models with $O\lt(\log \frac{1}{r}\rt)$ running time complexity. The method of \citet{Gj18} is a continuous version of the PRS algorithm. It has the same order of expected number of iterations as our algorithm, but  restricted to hard-core processes. One of our simulation results provides a comparison between the expected number of iterations of the proposed method and the method of \citet{Gj18} for a hard-core process.\\

The remaining paper is organized as follows:  In Section~\ref{sec:notation}, we introduce some notations that are useful throughout the paper. Section~\ref{sec:spp} provides definitions of the spatial point processes of interest.
In Section~\ref{sec:PRS}, the PRS method is presented and illustrated its application with an example. In Section~\ref{sec:PRS_GPP}, we propose our new perfect sampling method for Gibbs processes using PRS, and in Section~\ref{sec:RTC}, we analyze its running time complexity. Simulation results for Strauss process and PSM model are presented in Section~\ref{sec:sim}. The paper is concluded in Section~\ref{sec:con}.

\section{Notation}
\label{sec:notation}
First, some notation. $\reals_+$ is the set of non-negative real numbers and $\intg_+$ is the set of non-negative integers. $\reals^d$ denotes the $d$-dimensional Euclidean space with the corresponding Euclidean norm  $\| \cdot \|$. The distance between any two sets $C, D \subseteq \reals^d$ is defined by $$\mathsf{Dist}(C, D) = \inf\lt\{ \|x - y \| : x \in C \text{ and } y \in D\rt\},$$ with
$\mathsf{Dist}(\varnothing, C) = \infty$, where $\varnothing$ denotes the empty set. 
We use $\mathrm{e}$ to denote $\exp(1)$.
For any $x \in \reals_+$, $\lfloor x \rfloor$ is the largest $n \in \intg_+$ such that $n \leq x$. 
For any two probability measures $\mu_1$ and $\mu_2$ that are defined on the same measurable space, we write $\mu_1 \ll \mu_2$ to denote that $\mu_1$ is absolutely continuous with respect $\mu_2$.
We write $X \sim \mu_1$ to indicate that the distribution of a random object $X$ is $\mu_1$. 
The distributions of a Bernoulli random variable with success probability $p$, a uniform random variable over $(0,1)$ and a Poisson random variable with mean $\lambda$ are denoted, respectively, by $\mathsf{Bern}(p)$, $\mathsf{Unif}(0,1)$ and $\mathsf{Poi}(\lambda)$. For any event $A$, $\mathbb{I}(A)$ is equal to $1$ if the event holds, otherwise it is equal to $0$.


\section{Spatial Point Processes}
\label{sec:spp}
Consider a finite measure $\nu$ on a Euclidean subset ${S \subseteq \reals^d}$ that is absolutely continuous with respect to the Lebesgue measure.
Let $\scG$ be the set of all finite sets on $S$, defined by
\[
\scG := \Big\{\mx = \{x_1, x_2, \dots, x_n\} : n \in \intg_+ \text{ and } x_i \in S, \forall i \leq n \Big\}, 
\]
where $n = 0$ corresponds to the empty set. We assume that the elements of $\scG$ are {\it simple}, that is, they do not have multi-points.
For any ${\mx \in \scG}$, ${\big|\mx_A\big|}$ denotes the cardinality of ${\mx_A := \mx\cap A}$.
A {\it point process} is a random element $\mX : \Omega \to \scG$. \\

\noindent
{\bf Poisson point process (PPP):}
A point process ${\mX}$ is called {\it Poisson} on $S$ with intensity measure $\nu$ if 
it satisfies the following two properties:
\begin{enumerate}
	\item[(i)] $|\mX_A| \sim \mathsf{Poi}(\nu(A))$ for any measurable $A \subseteq  S$ and
	\item[(ii)] $|\mX_{A_1}|,  \dots, |\mX_{A_n}|$ are independent if $A_1,  \dots, A_n$ are measurable disjoint subsets of $S$.
\end{enumerate}
A PPP is called $\kappa$-{\em homogeneous} if the intensity ${\nu(\mathrm{d}x) = \kappa\, \mathrm{d}x}$ for some constant $\kappa > 0$. \\

 In several scenarios, it is important to associate an independent mark with each point in a PPP to characterize the shape or type of the object at that point. A {\it marked} PPP on $S$ is a PPP such that each point has a (random) mark independent of all other points. The mark associated with a point can {\em depend} on the point. For example, a mark at a point denotes the radius of a circle centered at that point. A typical realization of a marked PPP with $n$ points is of the form $\mx = \{ (z_1, m_1), (z_2, m_2), \dots, (z_n, m_n)\}$, where $\{z_1, z_2, \dots, z_n\} \in \scG$ and $m_i$ is the mark associated with $z_i$ for $i = 1, \dots, n$. For such a marked configuration, we define $\mx_A = \{(z_i, m_i) \in \mx : z_i \in A \}$ for any $A \subseteq S$. If the mark  space is $M$, then it is easy to see that the marked PPP is a PPP on $S\times M$.\\
 
It is common approach in statistical physics to wrap $S$ on a torus (that is, $S$ has periodic boundary) when large interacting particle systems are considered. In that case, throughout the paper, the Euclidean distance is replaced by geodesic distance.\\

\noindent
{\bf Gibbs point process:} 
Suppose that $\rho$ is the distribution of a (marked) PPP. A point process with distribution ${\mu \ll \rho}$ is called a {\em Gibbs point process} (or simply, Gibbs process) if the associated Radon-Nikodym derivative is of the form
\begin{equation}
 \label{eqn:abspos}
 \frac{\mathrm{d}\mu}{\mathrm{d}\rho} \lt(\mx \rt) = \frac{\exp\lt( - \mathcal{U}(\mx)\rt)}{Z},
\end{equation}
for every possible realization $\mx$ under $\rho$, where $\mathcal{U}$ is a non-negative real-valued potential function that is {\it non-degenerate} (i.e., ${\mathcal{U}(\{x\}) < \infty}$), and {\it hereditary} (i.e., ${\mathcal{U}(\mx) \leq \mathcal{U}(\mx')}$ for all $\mx \subseteq \mx'$). The normalizing constant $Z$ is equal to $\ee_{\rho} \lt[\exp\lt( - \mathcal{U}(\mX)\rt) \rt]$.\\

\noindent
{\bf Pairwise interaction point (PIP) processes:} A pairwise interaction point (PIP) process is a Gibbs point process for which the potential function is of the form 
\begin{align}
\label{eqn:Potn_pw}
\mathcal{U}(\mx ) = \sum_{\{x, y\} \subseteq \mx} f(x,y),\,\, \mx \in \scG,
\end{align}
where $f :\reals^d \times \reals^d \rightarrow \reals_+ \cup \{\infty\}$ is called  the pairwise interaction function; see, for e.g.,  \citet{CSKM13}. We say that a PIP has {\em finite}  range interaction if there exists $a < \infty$ such that $f(x, y) = 0$ for all $x, y \in S$ for which $\|x - y\| \geq a$; that is, the interaction between any two points is zero if they are separated by a distance of at least $a$. The smallest such $t$ is called the {\em interaction range} of the PIP. Some important PIP processes are considered below.\\

{\it Hard-core process:} A hard-core process with hard-core distance $2r > 0$ (that is, the hard-core radius is $r$) has
\[
f(x,y) = \begin{cases}
        \infty, &\text{ if }\, \|x - y\| < 2r,\\
        0,      &\text{ otherwise}.
        \end{cases} 
\]
 In a hard-core process no two points are within a distance of $2r$. Note that the interaction range here is $2r$. One generalization of the hard-core process is {\it hard-sphere} model with random radii, where the centers of spheres with independent and identically distributed ($i.i.d.$) radii constitute a PPP on $S$ conditioned on the event that no two spheres overlap.\\

{\it Strauss process:} Another well-studied PIP process is the {\it Strauss} process with parameters $\gamma \in [0,1]$ and $r > 0$. Here the interaction function is defined by
\[
f(x,y) = \begin{cases}
        -\log \gamma, &\text{ if }\, \|x - y\| < 2r,\\
        0,      &\text{ otherwise}.
        \end{cases} 
\]
The interaction range of this PIP process is $r$. This process becomes a hard-core process if $\gamma = 0$ with the convention that $0^0 = 1$.\\

{\it Strauss-hard core process:} This PIP process is a hybrid of the Strauss and hard-core processes, and has interaction function
\[
f(x,y) = \begin{cases}
        \infty, &\text{ if }\, \|x - y\| < a_1,\\
        -\log\gamma, &\text{ if }\, a_1 \leq \|x - y\| < a_2,\\
        0,      &\text{ otherwise},
        \end{cases} 
\]
for some $\gamma \in [0,1]$ and $0 < a_1 < a_2$. Here $t_1$ is called hard-core distance. Clearly, the interaction range for this process is $t_2$. \\

\noindent
{\bf Penetrable spheres mixture (PSM) model:} This model was introduced by \citet{WR70} to study liquid-vapor phase transitions.
Let $\rho $ is the distribution of $\kappa$-homogeneous marked PPP, where each point is independently marked either as type-1 (with probability $\kappa_1/(\kappa_1 + \kappa_2)$) or as type-2 (with probability $\kappa_1/(\kappa_1 + \kappa_2)$), for some constants $\kappa_1, \kappa_2 \geq 0$. A realization of a PSM model can be viewed as a realization of  $\mX \sim \rho$ conditioned on the event that no two points from different types are within a distance $2r$ from each other; that is, the corresponding potential function is given by \eqref{eqn:Potn_pw} with 
\[
f(x,y) = \begin{cases}
        \infty, &\text{ if }\, \|x - y\| < 2r \text{ and } x, y \text{ have different marks},\\
        0,      &\text{ otherwise}.
        \end{cases} 
\]

\noindent
{\bf Area-interaction process:} This process was first studied by  \citet{BV95} (see, also, \citet{KM00}, \citet{FFG02} and \citet{MJ01}). 
For any $A \subseteq \reals^d$, let $\mathsf{Vol}(A)$ be the volume of $A$ and $\mathsf{Ball}(x, a)$ be the $d$-dimensional sphere centered at $x$ with radius~$a$. The distribution of an area-interaction process on $S$ is absolutely continuous with respect to that of a $\lambda$-homogeneous PPP for some $\lambda > 0$, with the potential function given by 
\begin{align}
\label{eqn:pot_aip}
\mathcal{U}(\mx) = \beta\,\mathsf{Vol}\lt(\mathlarger{\cup}_{x \in \mx} \mathsf{Ball}(x, 2r)\rt),\, \, \mx \in \scG,
\end{align}
where the constant $\beta > 0$ is called {\em inverse temperature}; see Figure~\ref{fig:AIP} (a). The definition of area-interaction process given in \citet{BV95} is more general, as it allows $\beta < 0$. However, in this paper, we focus only on the case $\beta > 0$.\\

There is an interesting connection between area-interaction processes and PSM models. To see this, instead of \eqref{eqn:pot_aip}, if we suppose that the potential function is 
\begin{align}
\label{eqn:pot_aip_alt}
\mathcal{U}(\mx) = \beta\,\mathsf{Vol}\lt(S \mathlarger{\cap} \lt(\mathlarger{\cup}_{x \in \mx} \mathsf{Ball}(x, 2r)\rt)\rt),\, \, \mx \in \scG,
\end{align}
then the distribution of this modified area-interaction process is identical to the distribution of type-$1$ points of the PSM model with $\kappa = \lambda + \beta$, $\kappa_1 = \lambda$ and $\kappa_2 = \beta$; see Figure~\ref{fig:AIP} (b). This is because from the property~(i) in the definition of PPPs, 
for any $\mx~\in~\scG$, the probability that none of the points of a realization of a $\beta$-homogeneous PPP falls within the set $S \cap \lt(\mathlarger{\cup}_{x \in \mx} \mathsf{Ball}(x, 2r)\rt)$ is equal to $\exp\lt(-\beta\,\mathsf{Vol}\lt(S \mathlarger{\cap} \lt(\mathlarger{\cup}_{x \in \mx} \mathsf{Ball}(x, 2r)\rt)\rt)\rt)$. Further interesting fact is that if $S$ is periodic, both \eqref{eqn:pot_aip} and \eqref{eqn:pot_aip_alt} are the same. Hence, under the periodic assumption, an area-interaction process can be viewed as a realization of one type of points of a PSM model, and vice versa. This is the reason why area-interaction processes are sometimes referred as PSM models.\\

In the definition of PSM models, type-$1$ and type-$2$ points are independent PPPs on $S$ with intensities $\kappa_1$ and $\kappa_2$, respectively. Instead, if we assume that the type-$2$ points constitute a $\kappa_2$ homogeneous PPP on a bigger set $S(r)$ such that 
\[
\mathlarger{\cup}_{x \in S} \mathsf{Ball}(x, 2r) \subseteq S(r)
\]
(when $S = [0,1]^d$, $S(r)$ can be $[-2r, 1 + 2r]^d$). Then, with the choice of $\kappa_1 = \lambda$ and $\kappa_2 = \beta$, the distribution of type-$1$ points of this modified PSM model is identical to the distribution of the area-interaction process.

\begin{figure*}[h!]
    \centering
    \begin{subfigure}[t]{0.5\textwidth}
        \centering
        \includegraphics[height=0.8\textwidth]{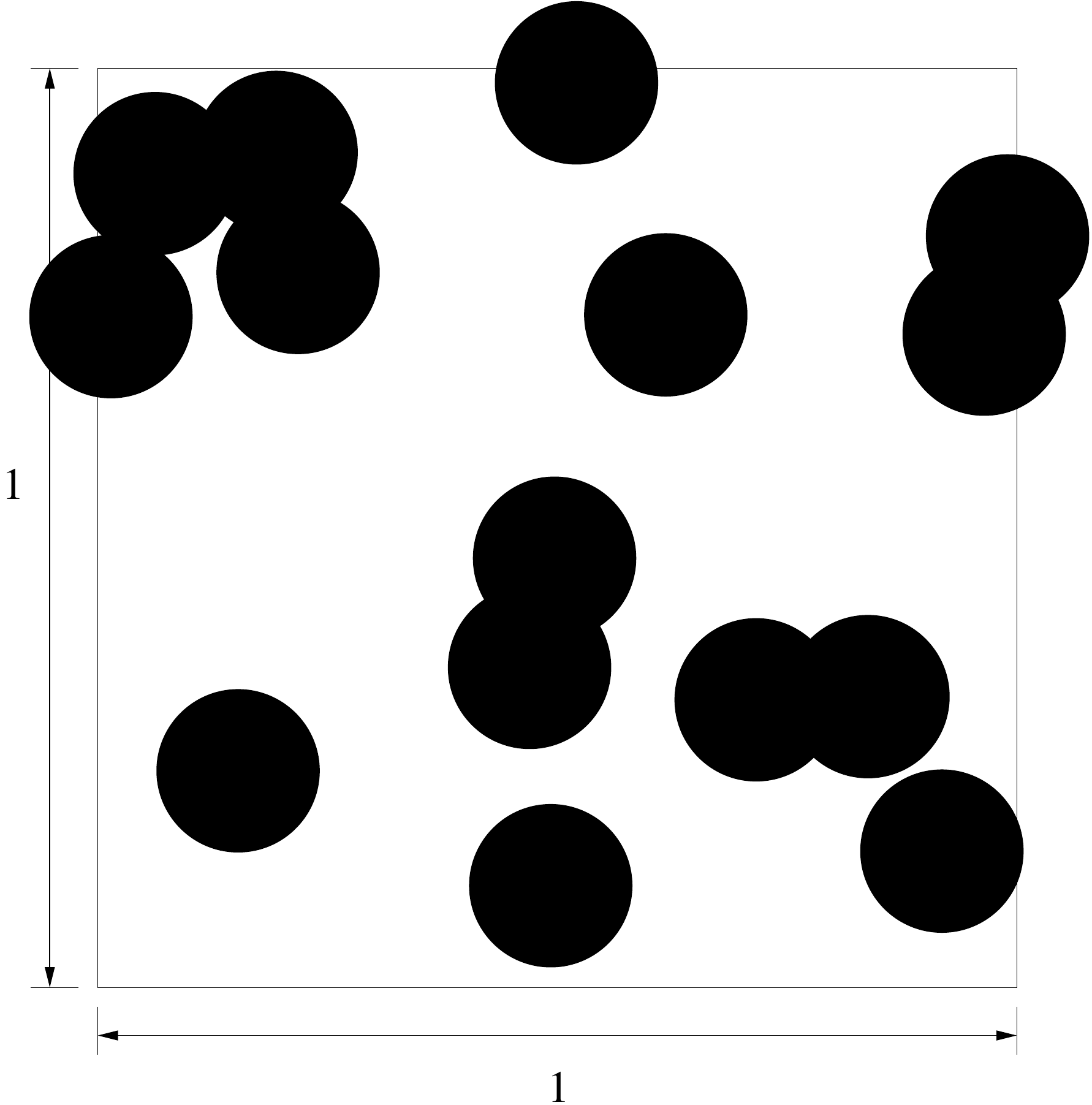}
        \caption{}
    \end{subfigure}%
    ~
    \begin{subfigure}[t]{0.5\textwidth}
        \centering
        \includegraphics[height=0.8\textwidth]{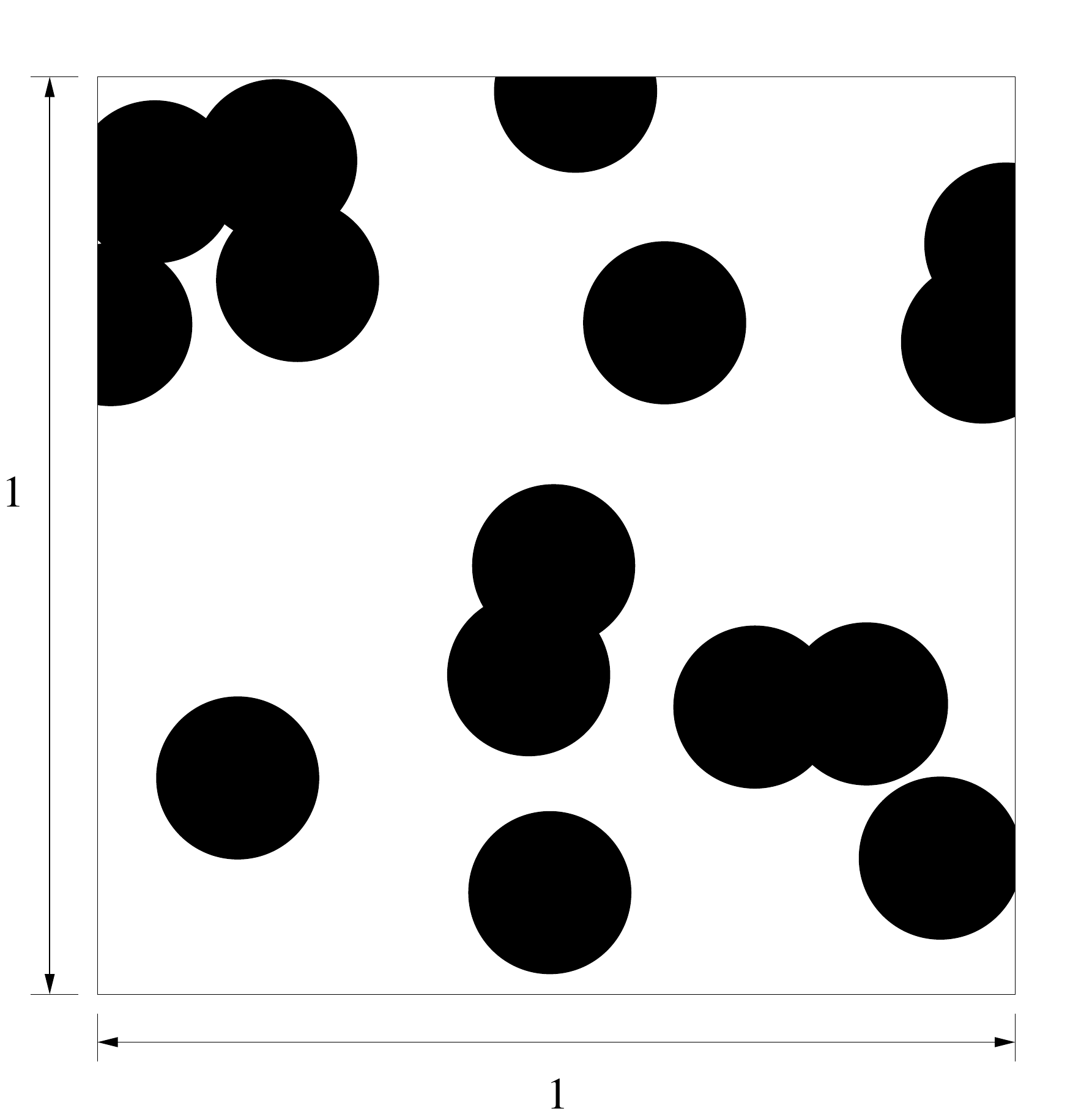}
        \caption{}        
    \end{subfigure}%
    \caption{\footnotesize A realization $\mx = \{x_1, \dots, x_{15}\}$ of an area-interaction process on a unit square $S = [0,1]^2$, where each $x_i$ is the center of a circle with radius $2r$. The dark regions in Panel (a) and Panel (b) are $\cup_{x \in \mx} \mathsf{Ball}(x,r)$ and $S \cap \lt(\cup_{x \in \mx} \mathsf{Ball}(x, 2r)\rt)$, respectively.}
    \label{fig:AIP}
\end{figure*}



\section{Partial Rejection Sampling}
\label{sec:PRS}
In this section, we briefly discuss the partial rejection sampling (PRS) method proposed in \citet[Section~4]{GJL17}. This method generates perfect samples from a product distribution, conditioned on none of a number of {\it bad} events happening. \\

To be precise, let $\mY = \{Y_1, Y_2, \dots, Y_n\}$ be a set of easy to simulate independent random objects on a probability space $(\varOmega, \mathcal{F}, \pp)$, taking values on $\mathcal{Y}$. Suppose that $\mu^{\otimes}$ is the distribution of $\mY$. Clearly,  $\mu^{\otimes}$ is a product distribution. Without loss of generality, we assume that $\mathcal{Y}$ is the support of $\mu^{\otimes}$. Let $\{{B_v \in \mathcal{F}} : v \in V\}$ be a set of bad events indexed by elements of a finite set~$V$.  Each bad event $B_v$ depends on a subset of $\mY$. Let $\mathcal{I}(v) \subseteq \{1,2,\dots,n\}$ be the largest set such that $B_v$ is dependent on $Y_i$ for all $i \in \mathcal{I}(v)$; that is, $\mathcal{I}(v)$ is the smallest set such that the set of variables $\{Y_i : i \in \mathcal{I}(v)\}$ imply whether the event $B_v$ occurs or not. By definition, $B_v$ does not depend on $\lt\{Y_i : i \in \{1, \dots, n \}\setminus \mathcal{I}(v) \rt\}$.  The goal of PRS is to generate perfect samples from $\mu^\otimes$, conditioned on the event that none of the bad events $\{B_v : {v \in V}\}$ occur. \\

One can generate the desired samples using the naive rejection sampling algorithm: repeatedly generate a sample from $\mu^\otimes$ until none of the bad events occur. The last sample has the desired distribution. In each iteration of this naive method, a fresh copy of the entire set $\mY$ is generated. Whereas, as we see below, in each iteration of the PRS method, only a subset of $\mY$ is resampled based on which bad events  occurred  in the previous iteration. This helps to significantly reduce the running time complexity compared with naive rejection sampling.\\

For any $ {u, v \in V}$, we write $u \leftrightarrow v$ if $\mathcal{I}(u) \cap \mathcal{I}(v) \neq \varnothing$. Define the {\em dependency graph} $G = (V, E)$ to be the graph, with the vertex set $V$ and edge set $E$ given by $$E = \lt\{ \{u, v\} : u, v \in V, u \neq v\, \text{ and }\, u \leftrightarrow v \rt\}.$$ That is, there is an edge between two nodes in $V$ if the bad events associated with the nodes depend on at least one common random object. 
For $\omega \in \varOmega$, let $$\mathsf{Bad}(\mY(\omega)) = \lt\{v \in V : \omega \in B_v \rt\}.$$ For any subset ${W \subseteq V}$, let $\partial W$ be the boundary of the set $W$ defined by 
$$\partial W = \{ v \in V: v \notin W \text{ and } \exists\, u \in W \text{ such that } u \leftrightarrow v\}.$$ 
Also define $\displaystyle \mathcal{I}(W) = \cup_{u \in W}\,  \mathcal{I}(u)$, and for any assignment $\my = \{y_i : i = 1, \dots, n \} \in \mathcal{Y}$ of $\mY$,  
let $\my|_W := \{y_i(\omega) : i \in \mathcal{I}(W) \}$ denote the partial assignment of $\my$ restricted to $\mathcal{I}(W)$. For any two assignments $\my, \my'$ of $\mY$, if $\my|_W = \my'|_W$ then $\my'$ is called an {\em extension} of $\my|_W$. Furthermore, an event $B$ is said to be {\em disjoint} from $\my|_W$ if either $\mathcal{I}(v) \cap \mathcal{I}(W) = \varnothing$ or $B$ can not occur for any extension of~$\my|_W$. \\

Algorithm~\ref{alg:PRS} generates a perfect sample ${\mY \sim \mu^\otimes}$, conditioned on the event that none of the bad events $B_v$ occurs. In each iteration of Algorithm~\ref{alg:PRS}, the inner while-loop constructs the {\em resampling} set ${\mathsf{Res} \subseteq V}$. It starts with ${\mathsf{Res} = \mathsf{Bad}(\mY)}$ where the initial assignment of random objects is~$\mY$, and recursively adds to $\mathsf{Res}$ the set of all the boundary vertices that are not disjoint from $\mathsf{Res}$, until there are no more boundary vertices to add. The final $\mathsf{Res}$ is the resampling set, and all the random objects with indices in $\bigcup_{u \in \mathsf{Res}} \mathcal{I}(u)$ are resampled. This construction is deterministic, in the sense that the final resampling set is a deterministic function of $\mY$. \\

\begin{algorithm}[H]
 \caption{PRS}
 \label{alg:PRS}
\SetAlgoLined
 Simulate $Y_1, Y_2, \dots, Y_n$ independently\\
 $\mY \leftarrow \{Y_1, Y_2, \dots, Y_n\}$\\
 \While{$\mathsf{Bad}(\mY) \neq \varnothing$}{
 $\mathsf{Res} \leftarrow \mathsf{Bad}(\mY)$ and $N \leftarrow \varnothing$\\
 \While{${\partial \mathsf{Res} \setminus N} \neq \varnothing$}{
  Let $D = \{v \in \partial \mathsf{Res} \setminus N :  B_v \text{ is disjoint from } \mY |_{\mathsf{Res}}\}$\\
  $N \leftarrow N \cup D$\\
  $\mathsf{Res} \leftarrow {\mathsf{Res}\cup\lt( \partial \mathsf{Res}\setminus N\rt)}$
 }
 Resample only the objects in $\lt\{ Y_i : i \in \mathcal{I}(\mathsf{Res})\rt\}$\\
 }
 Output $\mY$\\
\end{algorithm}
\ \\
We refer to \citet{GJL17} for a proof of correctness of the algorithm. We must note that in \citet{GJL17}, each $Y_i$ is a real-valued random variable, where as in this paper, we allow a more general setting by treating each $Y_i$ as a random object. However, the correctness proof still holds true for this general case as well.

\begin{example}[\em Hard-core model on a lattice]
\label{exp:hcm}
\normalfont
To illustrate the PRS algorithm, consider the following hard-core model  defined on a square lattice. 
Each node $i$ of the lattice is associated with an independent Bernoulli random variable $Y_i \sim \mathsf{Bern}\lt(\frac{\lambda}{1 + \lambda}\rt)$ for some $\lambda > 0$. The node $i$ is said to be occupied if $Y_i = 1$. Associated with each edge $\{i,j\}$, there is a bad event $B_{i,j}$ which holds if both the endpoints $i$ and $j$ are occupied; see Figure~\ref{fig:HC}. If we let $\mu^\otimes$ be the distribution of $Y_i$'s, then using the PRS algorithm we can generate a sample $\mY \sim \mu^\otimes$ conditioned on none of the bad events occurring. \\

\begin{figure*}[h!]
    \centering
        \begin{subfigure}[t]{0.33\textwidth}
        \centering
        \includegraphics[height=1.5in]{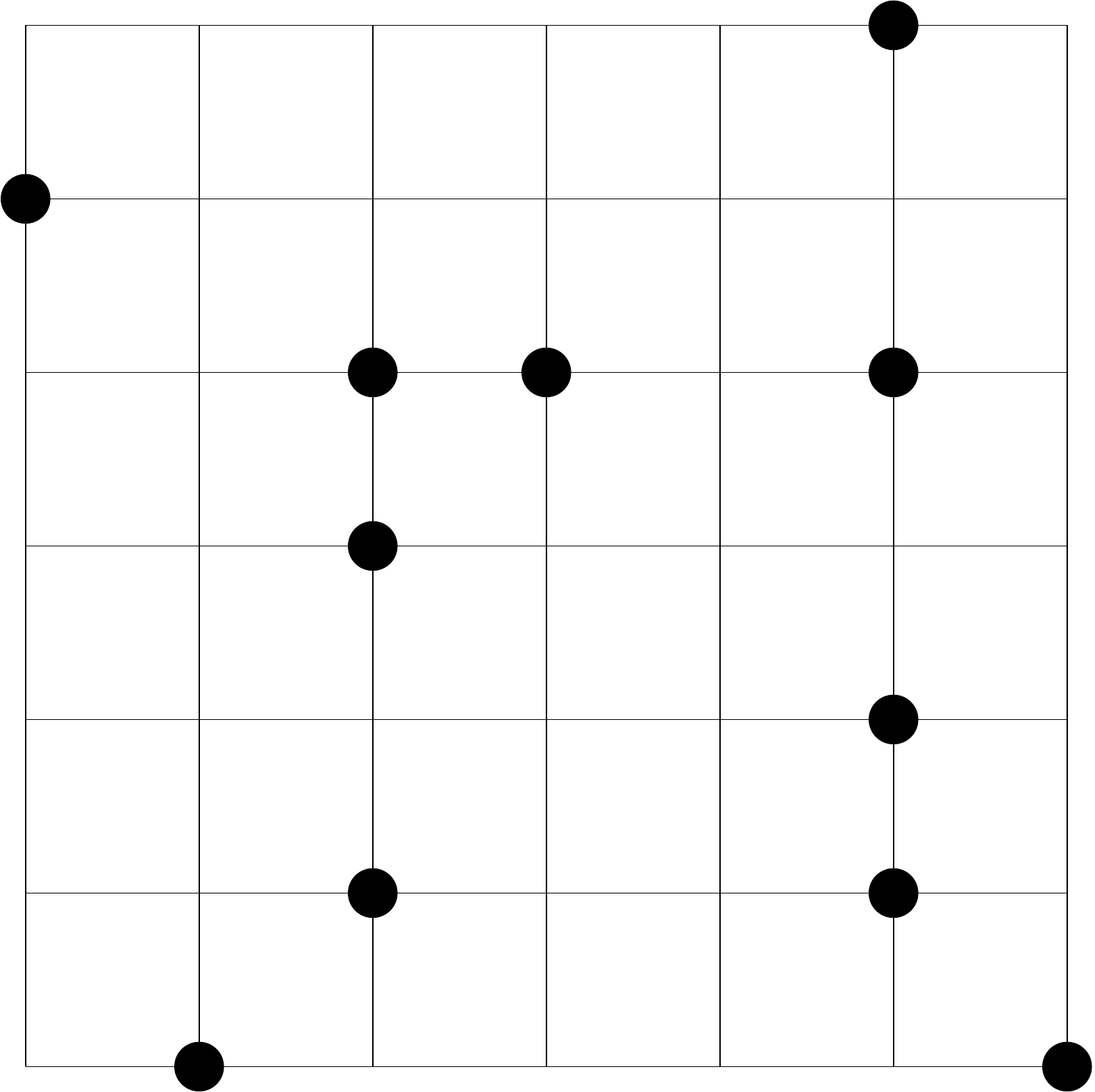}
        \caption{}
    \end{subfigure}%
    ~
    \begin{subfigure}[t]{0.33\textwidth}
        \centering
        \includegraphics[height=1.5in]{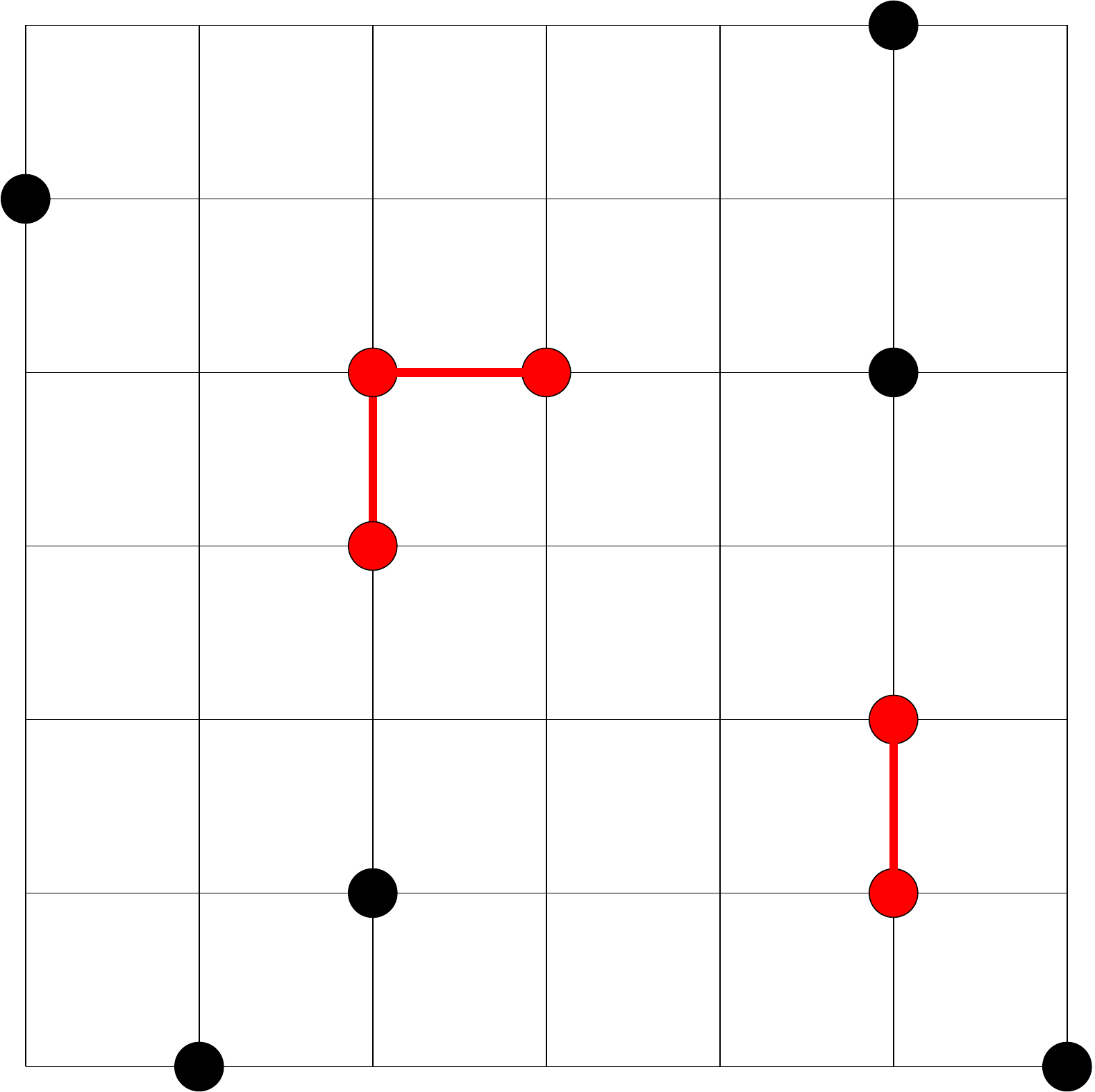}
        \caption{}        
    \end{subfigure}%
    ~ 
    \begin{subfigure}[t]{0.33\textwidth}
        \centering
        \includegraphics[height=1.5in]{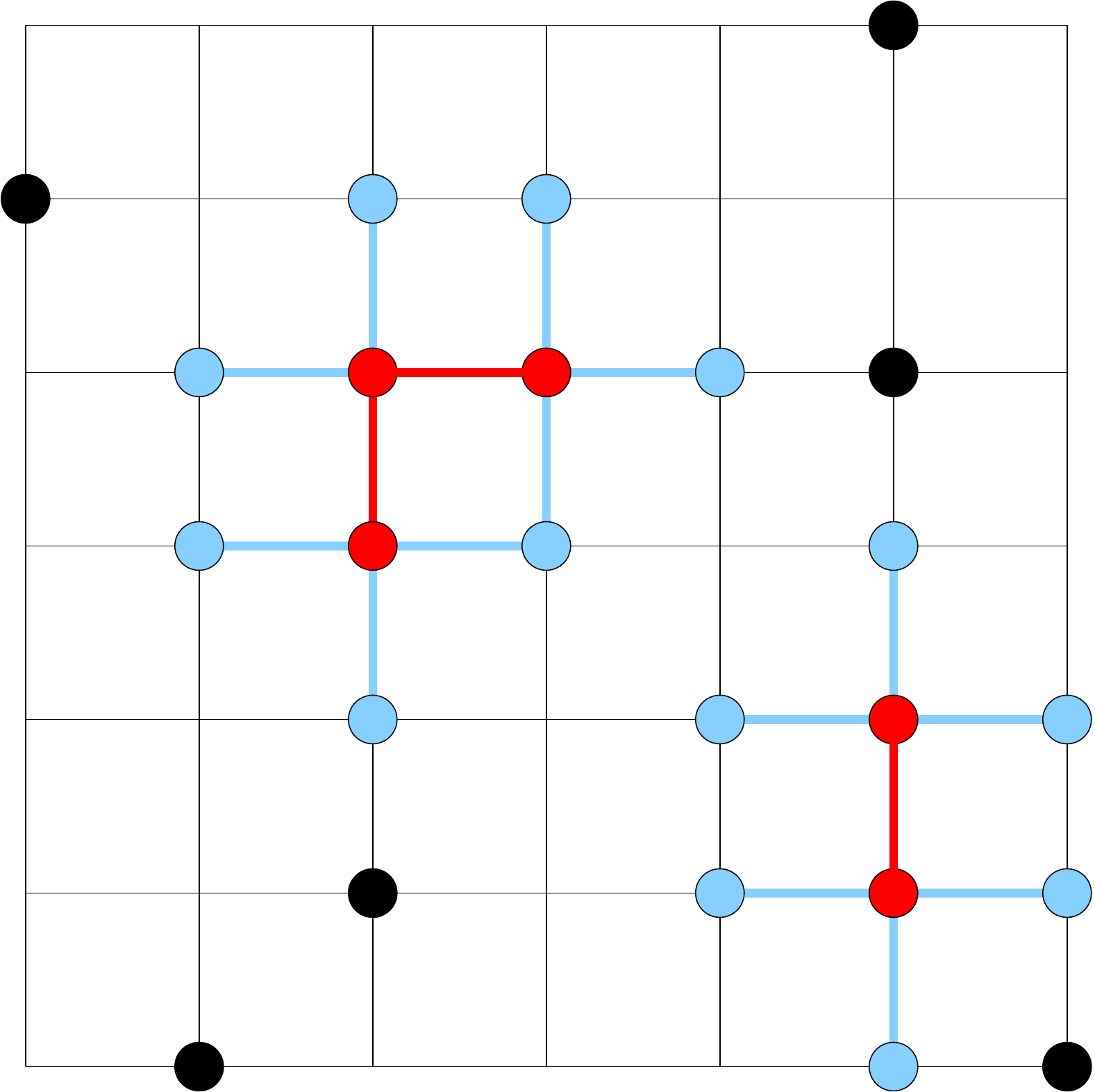}
        \caption{}
    \end{subfigure}
    \caption{\footnotesize Hard-core model on a square lattice. Panel (a): A typical realization of the hard-core model, where the occupied nodes are denoted by dark circles. Panel (b): The red coloured edges denote the bad events. Panel (c): The red and blue coloured edges together are the final resampling set; all the red and blue nodes will be resampled.}
    \label{fig:HC}
\end{figure*}

The corresponding dependency graph $G = (V, E)$ consists of $V$, the set of all the edges in the lattice, and 
\[ 
E = \lt\{ \{v,u\}: v \neq u, v \text{ and } u \text{ are connected by a common node}\rt\}.
\]
Clearly, if $v = \{i,j \} \in V$, we have $\mathcal{I}(v) = \{i,j\}$. As shown in Section~6.2 of \citet{GJL17}, for this hard-core model, it is easy to find the resampling set for any given set of bad events. Suppose at an iteration of the algorithm, if $\mathsf{Bad}$ is the set of bad edges of the lattice, then the resampling set $\mathsf{Res}$ is the union of $\mathsf{Bad}$ and $\partial \mathsf{Bad}$, where one endpoint of each edge in $\partial \mathsf{Bad}$ is {\em not} occupied and the other endpoint is shared with one of the edges in $\mathsf{Bad}$; see Figure~\ref{fig:HC}. 
\qed
\end{example}

As mentioned earlier, in general, generating a sample from $\mu^\otimes$ conditioned on none of the bad events happening is {\bf NP}-hard. However, under some additional conditions, Algorithm~\ref{alg:PRS} can be efficient in the sense that the expected number of iterations of the algorithm can be $O(\log |V|)$; see \citet{GJL17}. 
Lemma~\ref{lem:supermartingale} deals with one such case. Refer to \citet[Section~5]{GJL17} for a proof. Let $A_{ u, v}$ be the event that the partial assignment on $\mathcal{I}(v) \cap \mathcal{I}(u)$ can be extended to make $B_v$ occur, that is, with $W = \mathcal{I}(u) \cap \mathcal{I}(v)$,
$$A_{u,v} = \lt\{\omega \in \varOmega : \exists\, {\omega' \in B_v}  \text{ such that } \mY(\omega)|_W =  \mY(\omega')|_W\rt\}.$$
In particular, if $\{u,v\} \in E$ then $u \leftrightarrow v$ implies that $W = \mathcal{I}(u) \cap \mathcal{I}(v) \neq \varnothing$ and hence $A_{u, v}$ is the set of $\omega$ for which $B_v$ is not disjoint from $\mY(\omega)|_W$.\\

Define $\displaystyle p = \max_{v \in V } \pp_{\mu^{\otimes}}(B_v)$ and $\displaystyle q = \max_{\{u,v\} \in E} \pp_{\mu^{\otimes}}(A_{u.v})$. Let $\mathsf{Bad}_t$ and $\mathsf{Res}_t$ be, respectively, the set of bad vertices and the resampling set at iteration $t \geq 0$ of Algorithm~\ref{alg:PRS}. Further let $\Delta$ be the maximum degree of the dependency graph $G$.\\

\begin{lemma}[Lemma~5.4 of \citet{GJL17}]
\label{lem:supermartingale}
 For any ${\Delta \geq 2}$, if ${6 \mathrm{e} p \Delta^2 \leq 1}$ and ${3 \mathrm{e} q \Delta \leq 1}$, then  for all $t \geq 0$, $$\ee\lt[|\mathsf{Res}_{t+1}| \big| \mathsf{Res}_0, \dots, \mathsf{Res}_t  \rt] \leq (1 - p) |\mathsf{Res}_t|.$$
 
\end{lemma}
Note that $\mathsf{Bad}_t \subseteq \mathsf{Res}_t$ for all $t \geq 0$. From Lemma~\ref{lem:supermartingale} and the fact that $|\mathsf{Bad}_0| = |\mathsf{Res}_0| = |V|$ (since the algorithm starts with a fresh copy of all the random elements), 
\begin{align}
\label{eqn:bdd_on_BE}
\ee\lt[|\mathsf{Bad}_t|\rt] \leq \ee\lt[|\mathsf{Res}_t|\rt] \leq (1 - p)^t |V|,
\end{align} for all $t \geq 0$, under the hypothesis of the lemma. These observations are useful for the running time complexity analysis in Section~\ref{sec:RTC}.


\section{Perfect Sampling for Gibbs Point Processes}
\label{sec:PRS_GPP}
In this section, we propose a methodology to use the PRS algorithm for generating perfect samples of the Gibbs processes defined in Section~\ref{sec:spp}. For this, we partition the underlying space $S$ and using this partition, we identify certain bad events such that the target distribution can be expressed as a product distribution conditioned on none of these bad events occurring. For the case where $S  = [0,1]^d$, we consider a cubic-grid partitioning and specialize the PRS algorithm. \\

Recall the definition of Gibbs process with distribution $\mu$, given in Section~\ref{sec:spp}.  Assume that the corresponding potential function $\mathcal{U}$ has a finite interaction range $2r$ and 
\begin{align}
\label{eqn:gen_GPS}
\mathcal{U}\lt(\mx \rt) = \sum_{ \{x, y\} \subseteq \mx} f(x,y),\,\,\, \mx \in \scG,
\end{align}
for a function $f : S \times S \rightarrow \reals_+ \cup \{\infty \}$ such that $f(x, y) = f(y,x)$. Recall that $\mu \ll \rho$, with $\rho$ being the distribution of a (marked) PPP.
Clearly, both PIP process or PSM model (defined in Section~\ref{sec:spp}) can be seen as special cases of the above description. \\

Suppose $\lt\{C_1, C_2, \dots, C_n\rt\}$ is a partition of $S$ (i.e., the $C_i$'s are mutually disjoint and $\cup_{i = 1}^n C_i = S$). 
Let $V = \lt\{ v = \{i,j\} : \mathsf{Dist}(C_i, C_j) < 2r \text{ and } i \neq j\rt\}$ be the set of unordered pairs $\{i,j\}$ such that points that fall in $C_i$ can interact with points in $C_j$ and vice versa. As a consequence of \eqref{eqn:gen_GPS}, for any  $\mx \in \scG$,
\[
\mathcal{U}(\mx) = \sum_{i = 1}^n \mathcal{U}(\mx_{C_i}) + \sum_{\{i,j\} \in V } \sum_{\substack{x \in \mx_{C_i}\\  \, y\in \mx_{C_j}}} f(x, y).
\]
Hence,
\begin{align*}
\ee_{\rho}\lt[\exp\lt(- \mathcal{U}(\mX)\rt) \rt] &= \ee_{\rho}\lt[\prod_{i = 1}^n\exp\lt(- \mathcal{U}\lt(\mX_{C_i}\rt)\rt)  \prod_{\{i,j\} \in V } \exp\lt(- \sum_{x \in \mX_{C_i}, y\in \mX_{C_j}} f(x, y)\rt)  \rt]\\
              &= \ee_{\rho}\lt[\prod_{i = 1}^n\exp\lt(- \mathcal{U}\lt(\mX_{C_i}\rt)\rt)  \prod_{\{i,j\} \in V } \mathbb{I}\lt\{ U_{i,j} \leq \exp\lt(- \sum_{x \in \mX_{C_i},\, y\in \mX_{C_j}} f(x, y)\rt) \rt\} \rt],
\end{align*}
where $\{U_{i,j} : \{i,j\} \in V \}$ is a set of $i.i.d.$ $\mathsf{Unif(0,1)}$ random variables, independent of everything else.\\

For each $i$, let $\rho_i$ be the distribution of the reference (marked) PPP restricted to the cell $C_i$, that is,  if $\mX \sim \rho$ then $\rho_i$ is the distribution of $\mX_{C_i}$, and $\mX_{C_i}$ and $\mX_{C_j}$ are independent when $i \neq j$ (see the property (ii) in the definition of PPPs). 
Now let $\mu_i$ be the distribution of a Gibbs process on $C_i$ such that $\mu_i \ll \rho_i$ with the interaction range $2r$ and the potential function $\mathcal{U}\lt(\mx \rt) = \sum_{ \{x,y \} \in \mx } f(x,y)$ for all finite subsets $\mx \subseteq C_i$. This means that $\mu_i$ is the distribution of the target Gibbs process restricted to $C_i$.  Furthermore, define bad events 
$$B_{i,j} = \lt\{ \omega \in \Omega : U_{i,j}(\omega) > \exp\lt(- \sum_{x \in \mX_{C_i}(\omega),\,\, y\in \mX_{C_j}(\omega) } f(x, y)\rt) \rt\},\quad \{i,j\} \in V.$$
Let  $\rho = \rho_1\times \dots \times \rho_n$ and $\mu^\otimes := \mu_1 \times \mu_2 \times \dots \times \mu_n$. From the definition of $\mu$, $\mu \ll \mu^\otimes$ and
\begin{align*}
\frac{\mathrm{d} \mu}{\mathrm{d} \mu^\otimes}(\mx) = \frac{1}{\wt Z} \prod_{\{i,j\} \in V}\exp\lt( - \sum_{\{i,j\} \in V} \sum_{x \in \mx_{C_i},\, y\in \mx_{C_j}} f(x, y) \rt), \quad \mx \in \scG.
\end{align*}
Equivalently, $\mu$ is equal to the distribution $\mu^\otimes$ conditioned on none of the bad events $B_{i,j}$ happening. Here, the normalizing constant is $$\wt Z = \pp_{\mu^\otimes}\lt( \cap_{\{i,j \} \in V} B^c_{i,j}\rt) = \ee_{\mu^\otimes}\lt[  \prod_{\{i,j\} \in V} \exp\lt(- \sum_{x \in \mX_{C_i},\, y\in \mX_{C_j}} f(x, y)\rt)  \rt].$$

Since $\mu^\otimes$ is a product measure, if it is possible to generate samples from $\mu_i$'s, we can use the PRS, Algorithm~\ref{alg:PRS}, to generate samples from $\mu$.
The corresponding dependency graph is $G = (V, E)$, where, from the definition, $ E = \{\{u, v\} : u, v \in V \text{ and } u \leftrightarrow v\}$ with $B_{\{i, j\}} = B_{i,j}$.\\

To complete the argument, we need to spell out how to identify the resampling subset $\mathsf{Res}(\mY(\omega))$ of $V$ for every $\omega \in \varOmega$, where $$\mY := \lt\{\mX_{C_i}, \, i = 1,\dots, n, \text{ and } U_{i,j}, \, \{i,j\} \in V \rt\}.$$ This depends on knowing the condition for $B_{i,j}$ being disjoint from $\mY(\omega) |_W$ for all ${W \subseteq V}$ and all ${\{i,j\} \in \partial W}$. Lemma~\ref{lem:disjoint} establishes this condition. \\

To simplify the notion, we can take ${\mathcal{I}(\{i,j\}) = \{i,j\}}$ for all ${\{i,j\} \in V}$. That means, at any iteration of PRS, if ${\{i,j\} \in \mathsf{Res}}$ then $\mX_{C_i}$,  $\mX_{C_j}$ and $U_{i,j}$ will be resampled independently from their respective distributions. 
\begin{lemma}
\label{lem:disjoint}
Let  ${W \subseteq V}$ and  $\{j, k\} \in \partial W$ with ${j \in \mathcal{I}(W)}$. Then for any ${\omega \in \varOmega}$, $B_{j,k}$ is disjoint from the partial assignment $\mY(\omega) |_W$ if and only if $\mathsf{Dist}(\mX_{C_j}(\omega), C_k) \geq 2r$.
\end{lemma}
\begin{proof}
Observe that $k \notin \mathcal{I}(W)$, because $j \in \mathcal{I}(W)$ and $\{j,k\} \in \partial W$. Hence $\mathcal{I}(\{j,k\}) \cap \mathcal{I}(W) = \{j\}$. This implies that if $\mathsf{Dist}(\mX_{C_j}(\omega), C_k) \geq 2r$ then $B_{j,k}$ can not occur for any extension of $\{\mX_{C_j}(\omega) \}$, because no matter what is the configuration on $C_k$, it never interacts with the points of $\mX_{C_j}(\omega)$. On the other hand if $\mathsf{Dist}(\mX_{C_j}(\omega), C_k) < 2r$, we can find a $\omega' \in \Omega$ such that $\mX_{C_j}(\omega) = \mX_{C_j}(\omega')$ and $\omega' \in B_{j,k}$ (that is, the points of $\mX_{C_k}(\omega')$ interact with points of $\mX_{C_j}(\omega)$ to result in occurrence of $B_{j,k}$).
\end{proof}

\subsection{Cubic-grid partitioning}
\label{sec:cubic_grid}
Consider the PIP processes and PSM models defined in Section~\ref{sec:spp}. Suppose that $S = [0,1]^d$  is equipped  with a cubic grid of cell edge length $2r$. The area-interaction case is discussed at the end of this section. \\

\begin{figure*}[h!]
\centering
         \includegraphics[width=0.43\linewidth]{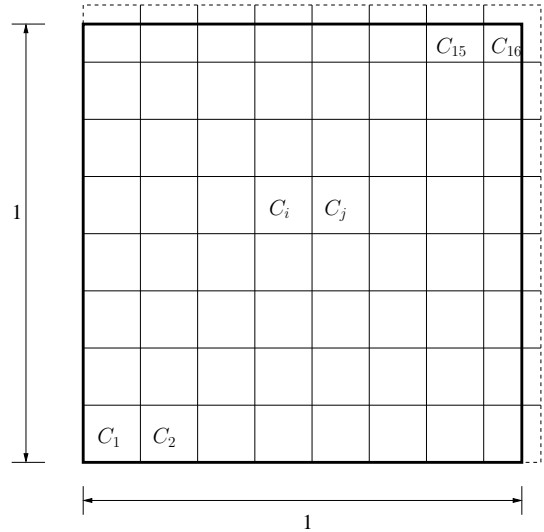}
  \caption{\footnotesize For a Gibbs process with interaction range $2r = 0.13$, the partition of the unit cube $S = [0,1]^2$ using a cubic grid of size $8 \times 8$ with the cell edge length $ 0.13$. Since $7\times 0.13 < 1 < 8\times 0.13$, some of the boundary cells are smaller in size than the other cells. Clearly, for any $i$ and $j$, $\{i, j\} \in V$ if and only if $C_i$ and $C_j$ are adjacent (i.e., $\mathsf{Dist}(C_i, C_j) = 0$). }
  \label{fig:grid}
\end{figure*}

So, the cells are $d$-dimensional cubes with volume $(2r)^d$, except some of the boundary cells that can be rectangular in shape with each edge is of length at most $2r$. When $2r = 1/K$ for some integer $K$, every  cell is a cube. We say that cells $C_i$ and $C_j$ are adjacent if $\mathsf{Dist}(C_i, C_j) = 0$. From this construction, it is evident that $\{ i,j\} \in V$ if and only if $C_i$ and $C_j$ are adjacent. Furthermore, there is no cross interaction between point configurations on two non-adjacent cells, that is, $${\mathcal{U}\lt(\mX_{C_i}(\omega) \cup \mX_{C_j} (\omega) \rt) = \mathcal{U}\lt(\mX_{C_i}(\omega)  \rt) + \mathcal{U}\lt(\mX_{C_j}(\omega) \rt)},$$ for all ${\omega \in \Omega}$ if $C_i$ and $C_j$ are not adjacent to each other; see Figure~\ref{fig:grid}. \\

Recall from Lemma~\ref{lem:disjoint} that the implementation of the PRS algorithm for Gibbs processes depends on deciding whether $\mathsf{Dist}(\mX_{C_i}, C_j) < 2r$, or not, for a given $\{i,j\} \in V$ and a point configuration $\mX_{C_i}$ on the $i^{th}$ cell. The  interesting aspect of this cubic grid partitioning is that for any $\{i,j\} \in V$, $\mathsf{Dist}(\mX_{C_i}, C_j) < 2r$ if and only if the point configuration $\mX_{C_i} = \varnothing$. To verify this claim notice that the 'if' part trivially follows from the definition of $\mathsf{Dist}$, and the 'only if' part follows from the observation that each cell has edges of length at most $2r$ and when $\mX_{C_i} \neq \varnothing$, we can select a point configuration on 
$C_j$ and a value for $U_{i,j}$ so that the bad event $B_{i,j}$ occur, making it not disjoint from $\mX_{C_i}$.
As a consequence, observe that for any realization of $\mY$ in an iteration of the PRS algorithm, $\mathsf{Res}$ is the minimal subset of $V$ such that $\mathsf{Bad} \subseteq \mathsf{Res}$ and 
\begin{align*}
\mX_{C_i} = \varnothing, \quad \text{for all } \, \{i, j\} \in \partial \mathsf{Res} \text{ with }\, i \in \mathcal{I}(\mathsf{Res}).
\end{align*} 

Under this setup, we now restate the PRS algorithm for Gibbs processes. We remind the reader that for each $i$, samples $\mX_{C_i}$ from $\mu_i$ are generated using any existing method such as dCFTP.\\

\begin{algorithm}[H]
 \caption{PRS using cubic-grid partitioning}
 \label{alg:PRS_GPP}
\SetAlgoLined
 Simulate $\mX_{C_i} \sim \mu_i$, $i = 1,\dots, n(r)$, and $U_{i,j} \sim \mathsf{Unif}(0,1)$, $\{i,j \} \in V$, independently\\
 $\mY \leftarrow  \lt\{\mX_{C_i}, \, i = 1,\dots, n, \text{ and } U_{i,j}, \, \{i,j\} \in V \rt\}$\\
 \While{$\mathsf{Bad}(\mY) \neq \varnothing$}{
 $\mathsf{Res} \leftarrow \mathsf{Bad}(\mY)$\\
 \Repeat{$D = \varnothing$}{
  Let $W = \lt\{ \{i,j \} \in \partial \mathsf{Res}:  i \in \mathcal{I}(\mathsf{Res}) \text{ and } \mX_{C_i} \neq \varnothing\rt\}$\\
  $\mathsf{Res} \leftarrow \mathsf{Res} \cup W$\\ 
  \vspace{2mm}
 }
 Resample $\mX_{C_i} \sim \mu_i$, $i \in \mathcal{I}(\mathsf{Res})$, and $U_{i,j} \sim \mathsf{Unif}(0,1)$, $\{i,j\} \in \mathsf{Res}$, independently\\
 }
 Output $\mY$\\
\end{algorithm}

\ \\
To generate perfect samples of an area-interaction process on $S = [0,1]^d$ with inverse temperature $\beta$ and the intensity of the reference PPP is $\lambda$, we use Algorithm~\ref{alg:PRS_GPP} to generate samples of the modified PSM model on ${S(r) = [-2r, 1 + 2r]^d}$ where the reference PPP consists of type-$1$ and type-$2$ points, with type-$1$ points being $\lambda$-homogeneous PPP on $S$ and type-$2$ points being $\beta$-homogeneous PPP on $S(r)$. Type-$1$ points in the output of the algorithm is a sample of the target area-interaction process. Refer Section~\ref{sec:spp} for the connection between area-interaction processes and PSM models. In Algorithm~\ref{alg:PRS_GPP}, instead of $S$, we equip $S(r)$ with a cubic-grid partitioning. On each cell $C_i$, $\mu_i$ is the distribution of a modified PSM model where the reference PPP consists of type-$1$ and type-$2$ points, with type-$1$ points being $\lambda$-homogeneous PPP on $C_i \cap S$ and type-$2$ points being $\beta$-homogeneous PPP on $C_i\setminus S$. 


\section{Running Time Analysis}
\label{sec:RTC}
In this section, we assume that $S = [0,1]^d$ and the target Gibbs distribution $\mu \ll \rho$, with potential function given by \eqref{eqn:gen_GPS} and interaction range $2r \leq 1$. We further assume that $\rho$ is the distribution of a $\kappa$-homogeneous (marked) PPP on $S$ with $\kappa = \kappa_0/(\mathsf{v}_d r^d)$ for some $\kappa_0 > 0$. We analyze the running time complexity of the partitioning based PRS (described in Section~\ref{sec:PRS_GPP}) as $r \to 0$. We further compare this method with two well-known existing methods by establishing trivial lower bounds on the running time complexities of the existing methods.\\

Suppose that $\{ C_1, C_2, \dots, C_{n(r)}\}$ is a partition of $S$ such that samples from $\mu_i$ can be simulated  using  any of the existing perfect sampling algorithms, such as the dCFTP  (see Section~\ref{sec:PRS_GPP} for the definition of $\mu_i$).  As shown in Figure~\ref{fig:grid}, one possible partitioning is a cubic grid. For each $i = 1,2, \dots, n(r)$,  let $N_i$ be the number of cells $C_j$, $j \neq i$, such that $\mathsf{Dist}(C_i, C_j) < 2r$. Theorem~\ref{thm:complex} below establishes that if the volume of each cell is chosen to be of order $r$ and the $N_i$'s are uniformly bounded for all $r$, then there exist a constant $\bar \kappa$ such that for all $\kappa_0 \leq \bar \kappa$, the expected number of iterations the PRS algorithm takes to generate a perfect sample is $O\lt(\log \lt(\frac{1}{r}\rt) \rt)$. Observe that for the cubic-grid partition of Subsection~\ref{sec:cubic_grid}, $N_i$ is bounded by a constant uniformly and ${\frac{\mathsf{Vol}(C_i)}{r^d} \leq 2^d}$, for all $r$ and $i$.  Hence the conditions in Theorem~\ref{thm:complex}~hold.\\

As mentioned in \citet{GJL17}, an interesting feature of the PRS algorithm is that it is distributive. In particular, if we assume that each cell $i$ is associated with a processor that can generate samples from $\mu_i$ and can communicate with other processors within a constant time, then as we argue later in the proof of Theorem~\ref{thm:complex}, using parallel programming, the expected running time complexity of finding the resampling set  in any iteration can be reduced to $O(1)$. In that case the expected running time complexity of the PRS algorithm is simply of order of the expected number of iterations, which is $O\lt(\log \lt(\frac{1}{r}\rt) \rt)$. See, for example, \citet{FY18, FSY17} for recent works on distributed sampling.

\begin{theorem}
\label{thm:complex} 
Suppose that there exists constants $a, b > 0$ such that for all $i = 1, \dots, n(r)$, $N_i \leq a$ and
\begin{align}
\label{eqn:hypo_thm}
\frac{\mathsf{Vol}(C_i)}{r^d} \leq b \, \, \text{ for all } \, r.
\end{align} 
Then there exists $\bar \kappa > 0$ such that for all $\kappa_0 \leq \bar \kappa$, we have
\begin{itemize}
\item[(i)] the expected number of iterations of the PRS algorithm is $O\lt(\log \lt(\frac{1}{r}\rt) \rt)$,
\item[(ii)] the expected running time complexity of the algorithm is $O\lt(\frac{1}{r^{d}} \log \lt( \frac{1}{r}\rt)\rt)$, and
\item[(iii)] if the implementation of the algorithm is distributive, then the expected running time complexity is $O\lt(\log \lt(\frac{1}{r}\rt) \rt)$.
\end{itemize}
\end{theorem}
\begin{proof}
The expected number of  points generated under $\rho$ within cell $i$ is $\kappa \mathsf{Vol}(C_i)$, which can be upper bounded by $\frac{\kappa_0\,b}{\mathsf{v}_d}$, under the assumption \eqref{eqn:hypo_thm}. Since the bound is independent of $r$, for any given $\kappa_0$, the running time complexity of generating a sample from $\mu_i$ is $O(1)$, for all $i$, using any standard perfect sampling algorithm; see Chapter~7 of \citet{MH16}. \\

Since $N_i \leq a$ for all $r$ and $i$, the total number of nodes $|V|$ of the dependency graph is of order $n(r)$ and the maximum degree $\Delta$ is uniformly bounded for all $r$. The lower bound in \eqref{eqn:hypo_thm} implies that $n(r)$ is of order $1/r^d$.\\

We first show that both $p$ and $q$ go to zero as $\kappa_0$ goes to zero, and then we prove $(i) - (iii)$ as a consequence of~\eqref{eqn:bdd_on_BE}. For any $\{i,j\} \in V$, occurrence of the event $B_{i,j}$ implies that both the cells $i$ and $j$ have non-empty configurations. Since $\mu_i \ll \rho_i$, from the definition of Gibbs process, the probability of cell $i$ has an empty configuration is $$\pp_{\mu^{\otimes}} \lt(\mX_{C_i} = \varnothing\rt) = \exp\lt(- \kappa \mathsf{Vol}(C_i)\rt)/\widetilde Z_i,$$ 
where $\widetilde Z_i = \ee_{\rho} \lt[ \exp\lt( - \mathcal{U}\lt(\mX_{C_i} \rt) \rt)\rt] = \ee_{\rho_i} \lt[ \exp\lt( - \mathcal{U}\lt(\mX \rt) \rt)\rt]$. 
Observe that for any ${\omega \in \Omega}$, if either ${\mX_{C_i} = \varnothing}$ or ${\mX_{C_j} = \varnothing}$ then $\omega \in B^c_{i,j}$. Hence, 
\begin{align}
\label{eqn:ub_p}
p  &\leq \max_{\{i,j \} \in V} \pp_{\mu^{\otimes}}\lt( \mX_{C_i} \neq \varnothing \text{ and  } \mX_{C_j} \neq \varnothing \rt) \nonumber\\
   &= \max_{\{i,j \} \in V} \lt[\pp_{\mu^{\otimes}}\lt( \mX_{C_i} \neq \varnothing\rt) \pp_{\mu^{\otimes}}\lt( \mX_{C_j} \neq \varnothing\rt)\rt]\nonumber\\
   &= \max_{\{i,j \} \in V} \lt[ \lt(1 -  \exp\lt(- \kappa \mathsf{Vol}(C_i)\rt)/\widetilde Z_i\rt) \lt(1 -  \exp\lt(- \kappa \mathsf{Vol}(C_j)\rt)/\widetilde Z_j\rt) \rt]\\
   &\leq \max_{\{i,j \} \in V} \lt[ \lt(1 -  \exp\lt(- \kappa \mathsf{Vol}(C_i)\rt)\rt) \lt(1 -  \exp\lt(- \kappa \mathsf{Vol}(C_j)\rt)\rt) \rt], \nonumber
\end{align}
where the last inequality holds because $\widetilde Z_i \leq 1$ for all $i$. Using \eqref{eqn:hypo_thm}, we write that
\begin{align*}
p &\leq \lt(1 -  \exp\lt(- \frac{\kappa_0\,b}{\mathsf{v}_d}\rt)\rt)^2.
\end{align*}
Therefore, $p$ goes to zero as $\kappa_0$ goes to zero.\\

Recall that $A_{u, v}$ is the event that the partial assignment on $\mathcal{I}(u) \cap \mathcal{I}(v)$ can be extended to make $B_u$ true. Also recall that if ${\{u,v\} \in E}$, there exists $i, j$ and $k$ such that $u = \{i,k\}$ and $v = \{j,k\}$. Thus,  $\mathcal{I}(u) \cap \mathcal{I}(v) = \{k\}$. This implies that the event $A_{v, u}$ can not occur if the common cell $k$ is empty. Thus, 
\begin{align}
\label{eqn:ub_q}
\pp_{\mu^\otimes}(A_{v, u}) &\leq  \lt(1 -  \exp\lt(- \kappa \mathsf{Vol}(C_k)\rt)/\widetilde Z_k\rt)\\
                                        &\leq  \lt(1 -  \exp\lt(- \kappa \mathsf{Vol}(C_k)\rt)\rt) \nonumber \\
                                        &\leq 1 - \exp\lt(- \frac{\kappa_0 \, b}{\mathsf{v}_d}\rt).\nonumber
\end{align}
As a consequence $q \leq 1 - \exp\lt(- \frac{\kappa_0 \,b}{\mathsf{v}_d}\rt)$, which goes to zero as $\kappa_0$ goes to zero. \\

Since the maximum degree $\Delta$ of the dependency graph  does not change with the value of $\kappa_0$, there exists a constant $\bar \kappa$ such that $6 \mathrm{e} p \Delta^2 \leq 1$ and $3 \mathrm{e} q \Delta \leq 1$ for all $\kappa_0 \leq \bar \kappa$. From \eqref{eqn:bdd_on_BE}, within an order of $\log |V|$ iterations the expected number of bad events is less than $1$. Since $|V|$ is $O(1/r^d)$, the expected number of iterations of the PRS algorithm is $O\lt(\log (1/r) \rt)$. This proves $(i)$.\\

Furthermore, since the number of random objects resampled at iteration $t$ of the algorithm is $|\mathsf{Res}_t|$, the expected running time of the algorithm is of order $\sum_{t = 0 }^\infty |\mathsf{Res}_t|$, which is less than $|V|/p$, by \eqref{eqn:bdd_on_BE}. Hence $(ii)$ is established. \\

In order to prove $(iii)$, it is enough to show using parallel computation that the expected complexity of constructing the resampling set in each iteration of the algorithm is $O(1)$. To show this, we suppose that starting from a bad event, using breadth-first search, we identify the resampling events associated with the bad event. This can be done in parallel starting from every bad event. Then the final resampling set is the union of all the resampling events identified. Note that for each bad event, we first add the boundary events of the bad event to the resampling set; the number of boundary events is at most $\Delta$. For each added event, on average at most $q\Delta$ events from its boundary events are added to the resampling set. This will go on until there are no more events to add. So, for each bad event, the number of resampling events added is bounded by 
$$\Delta[1  + q\Delta[1 +  q\Delta[1 + \cdots]]],$$ which is further bounded by $3 \mathrm{e} \Delta/(3 \mathrm{e} - 1)$ when $3 \mathrm{e} q\Delta \leq 1$. Since in each iteration, the resampling sets associated with the bad events are constructed in parallel, the expected running time complexity of constructing the final resampling set in each iteration is $O(1)$.
\end{proof}

\subsection{Comparison with existing well-known methods}
In this subsection, we consider two well-known methods, namely, the naive rejection sampling and the dCFTP methods, and establish a trivial lower bounds on their expected running time complexity. \\

\noindent
{\bf Naive rejection method:} For a Gibbs process of the form \eqref{eqn:abspos}, the naive rejection sampling method repeatedly simulates a sample $\mX$ from $\rho$ until it is accepted. The last sample has the target distribution $\mu$. The acceptance probability at each iteration is $Z = \ee_\rho\lt[\exp(- \mathcal{U}(\mX))\rt]$, and thus the expected number of iterations is  $1/Z$. Since the expected number of points generated in each iteration is $\kappa$ (because $\rho$ is $\kappa$-homogeneous), the expected running time of the naive algorithm is proportional to $\kappa/Z$. Below, we use a standard argument to show that $\kappa/Z$ increases faster than an exponential function as $r$ decreases to $0$.\\

Consider the cubic grid partitioning of Subsection~\ref{sec:cubic_grid}. Note that each cell is at most as big as a cube with the edge length $2r$ and the number of cells $n(r)$ is at least $\lceil 1/2r \rceil$. Therefore, by ignoring the cross correlations between the cells and using the fact that there are at least $(n(r) - 2)^d$ cubic cells, we obtain 
$$Z \leq \prod_{i =1}^{(n(r) - 2)^d} \ee_{\rho}[\exp(-\mathcal{U}(\mX_{C_i}))] \leq \lt(\ee_{\rho}[\exp(-\mathcal{U}(\mX_{C}))] \rt)^{\lt(\frac{1 - 4r}{2r}\rt)^d} =: \varepsilon^{\lt(\frac{1 - 4r}{2r}\rt)^d},$$
where ${C = [0,2r]^d}$, ${\varepsilon = \ee_{\rho}[\exp(-\mathcal{U}(\mX_{C}))]}$, and the inequalities hold because ${n(r) \geq 1/2r}$ and ${\varepsilon \leq 1}$.\\

Since for any fixed $\kappa_0 >0 $, the value of $\varepsilon$ is strictly less than $1$ and does not depend on~$r$ (because, $\varepsilon$ is the same if $C = [0,1]^d$, $\rho$ is the distribution of $(2^d \kappa_0/\mathsf{v}_d )$-homogeneous PPP on $C$, and the interaction range of the potential function $\mathcal{U}$ is $1$). By using the value of $\kappa$ and the upper bound on $Z$ we obtain a lower bound on the expected running time complexity $\kappa/Z$, given by
\[
\frac{\kappa}{Z} \geq \lt(\frac{\kappa_0}{\mathsf{v}_d r^d} \rt)\lt(\frac{1}{\varepsilon}\rt)^{\lt(\frac{1 - 4r}{2r}\rt)^d},
\]
which increases faster than an exponential function, as $r$ goes to $0$.
\\

\noindent
{\bf Dominated CFTP:} In order to establish a lower bound on the expected running time of a dominated CFTP method, we first briefly state the general description the method (for a detailed description, we refer the reader to, for example, \citet{KM00}). Let $\mD = \{\mD(t): t \in \reals\}$ be the (free) birth-and-death process on $S = [0,1]^d$ with birth rate $\kappa$, where each birth is a (marked) point uniformly and independently selected on $S$ and alive for a random time exponentially distributed with mean one. It is not difficult to show that the steady-state distribution of $\mD$ is $\rho$. Since the target Gibbs distribution $\mu \ll \rho$, using coupling, it is possible to construct a process $\mZ = \{\mZ(t) : t \in \reals\}$ such that $\mZ(t) \subseteq \mD(t)$ for all $t \in \reals$ and the steady-state distribution of $\mZ$ is $\mu$. Any dCFTP method consists of two steps: i) constructing the dominating spatial birth-and-death process 
$\{\mD(t): - s \leq t \leq 0\}$ backward in time, for some $s > 0$, starting at time zero with $\mD(0) \sim \rho$, and ii) use {\em thinning} on the dominating process to obtain an upper bounding process $\{\mU_s(t) : t \geq -s \}$ with $\mU_s(-s) = \mD(-s)$ and a lower bounding process $\{\mL_s(t) : t \geq -s\}$ with $\mL_s(-s) = \varnothing$, forward in time such that the condition $\mL_s(t)\subseteq \mZ(t) \subseteq \mU_s(t) \subseteq \mD(t)$ is guarantee to hold for all $t \geq -s$. If $\mU_s$ and $\mL_s$ {\em coalescence} at time $0$, that is, $\mU_s(0) = \mL_s(0)$, then  $\mU_s(0)$ is a perfect sample from the target distribution $\mu$. If there is no coalescence, then in the next iteration, increase~$s$ and extend the dominating process further backward to time $-s$ and repeat the same procedure. \\

The criteria for thinning depends on the definition of the target distribution. However, the dominating process depends only on $\rho$. 
Let $\widetilde T$ be the {\em backward coalescence} time given by $$\widetilde T := \min\{s \geq 0 : \mL_s(0) = \mU_s(0)\}.$$
The running time complexity of a dCFTP method is at least of order of the number of computations  needed to construct the dominating process $\{\mD(t) : -\widetilde T \leq t \leq 0\}$. 
Let $$T_s = \{ {t \geq 0}: \text{ none of the points in } {\mD(-s)} \text{ are alive at } {-s + t}\}.$$ 
Since the dominating process $\mD$ is time-reversible and $\mD(0) \sim \rho$, we have $\mD(-s) \sim \rho$ for all $s \geq 0$.
Hence, the distribution of $T_s$ does not depend on $s$.\\

Since $\mU_s(-s) = \mD(-s)$ and $\mL_s(-s) = \varnothing$, if a point of $\mD(-s)$ is alive at time $0$, then $\mU_s(0) \neq \mL_s(0)$. Therefore, $T_s \geq s$ implies that $\widetilde T \geq s$ and thus $s$ is at least $T_s$ for the coalescence to happen. As a consequence, the expected running time complexity of any dCFTP algorithm is at least of order of the expected number of births in $\mD$ that are generated during an interval of length $\ee[T_s]$, which is $\kappa\, \ee[T_s]$ because the births in the dominating process are Poisson with rate $\kappa$. \\

Further recall that each birth is alive for a random time independently and exponentially distributed with mean one. Conditioned on $|\mD(-s)| = m$,  $T_s$ is the maximum of $m$ $i.i.d.$ mean one exponential random variables. Since $|\mD(-s)| \sim \mathsf{Poi}(\kappa)$ for all $s$, we have
\begin{align*}
\ee[T_s] = \sum_{m = 0}^\infty \mathsf{e}^{- \kappa} \frac{\kappa^m}{m!} \ee\lt[T_s \big| |\mD(-s)| = m\rt] \geq \sum_{m \geq \kappa/2} \mathsf{e}^{- \kappa} \frac{\kappa^m}{m!} \ee\lt[T_s \big| |\mD(-s)| = m\rt] = \sum_{m \geq \kappa/2}^\infty \mathsf{e}^{- \kappa} \frac{\kappa^m}{m!} H(m),
\end{align*}
where ${H(m) = \sum_{i = 1}^m \frac{1}{i}}$ is the $m^{th}$ harmonic number. Using the fact that ${H(m) \geq \log m}$ for all ${m \geq 1}$, 
\begin{align*}
\ee[T_s] \geq \log(\kappa/2) \pp\lt(|\mD(-s)| \geq \kappa/2\rt).
\end{align*}
From Chernoff bound on the tail probabilities of Poisson distribution, there exists a constant $a > 0$ such that $\pp\lt(|\mD(-s)| \geq \kappa/2\rt) = 1 - \exp\lt( - a\kappa\rt)$, for all values of $\kappa$. In conclusion, the expected running time complexity of any dCFTP algorithm is at least of order of $ \kappa \log \kappa$, which is of order of $\frac{1}{r^d} \log\lt(\frac{1}{r} \rt)$ for any $\kappa_0$. 

\section{Simulations}
\label{sec:sim}
In this section, we take $S = [0,1]^2$ and apply Algorithm~\ref{alg:PRS_GPP} to generate perfect samples of hard-core process, Strauss process and PSM models. We ignore the case of area-interaction process as the implementation is similar to that of PSM model and expected to have same order of complexity. \\

We estimate the expected number of iterations of the algorithm for different values of the model parameters. As long as the samples on each cell are perfect, the reported results are the same for any choice of existing method to generate samples from $\mu_i$'s. \\

\noindent
{\bf Hard-core and Strauss processes:} Consider the Strauss process with intensity $\kappa = \kappa_0/(\mathsf{v}_d r^d)$. Recall that the hard-core process is a Strauss process with ${\gamma = 0}$. Panels (a), (b) and (c) in Figure~\ref{fig:sim_strauss} are correspond to $\kappa_0$ is equal to $0.1$, $0.2$ and $0.25$, respectively. Each panel has two curves corresponds to $\gamma = 0$ (i.e., hard-core process) and $\gamma = 0.5$. Each curve is the estimated expected number of iterations of the algorithm as a function of $K$, when the interaction range $2r = 1/K$. Perfect samples from $\mu_i$ on each cell $i$ are generated using the dCFTP method by \citet{MH12}.
Observe that in each case, the expected number of iterations of the algorithm seems to be $O\lt(\log(1/r)\rt)$ as shown in Theorem~\ref{thm:complex}. These results suggest that $\bar \kappa$ in Theorem~\ref{thm:complex} can be at least $0.25$. 

\begin{figure*}[h!]
    \centering
    \begin{subfigure}[t]{0.32\textwidth}
        \centering
        \includegraphics[height=0.8\textwidth]{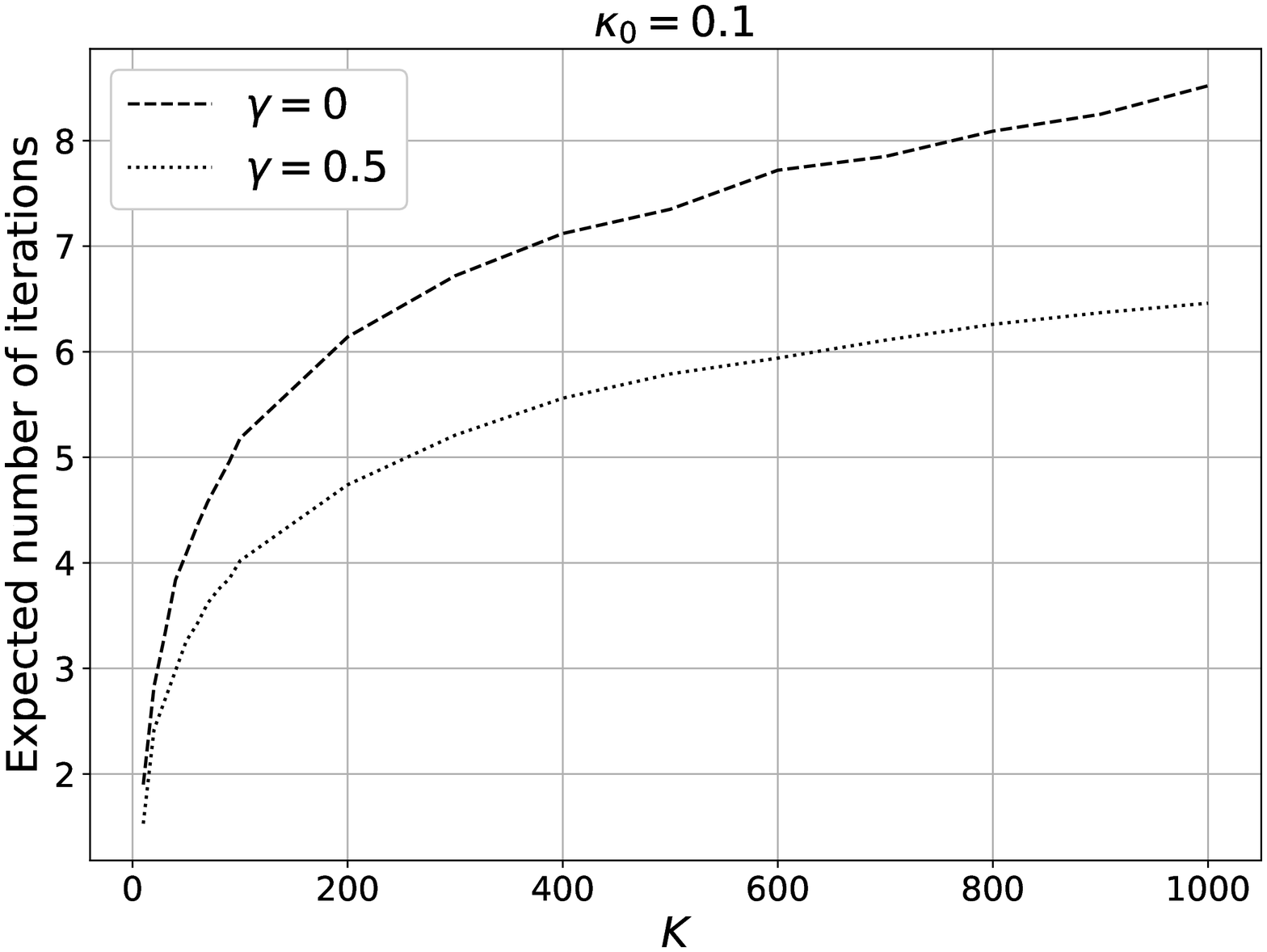}
        \caption{}
    \end{subfigure}%
  ~
    \begin{subfigure}[t]{0.32\textwidth}
        \centering
        \includegraphics[height=0.8\textwidth]{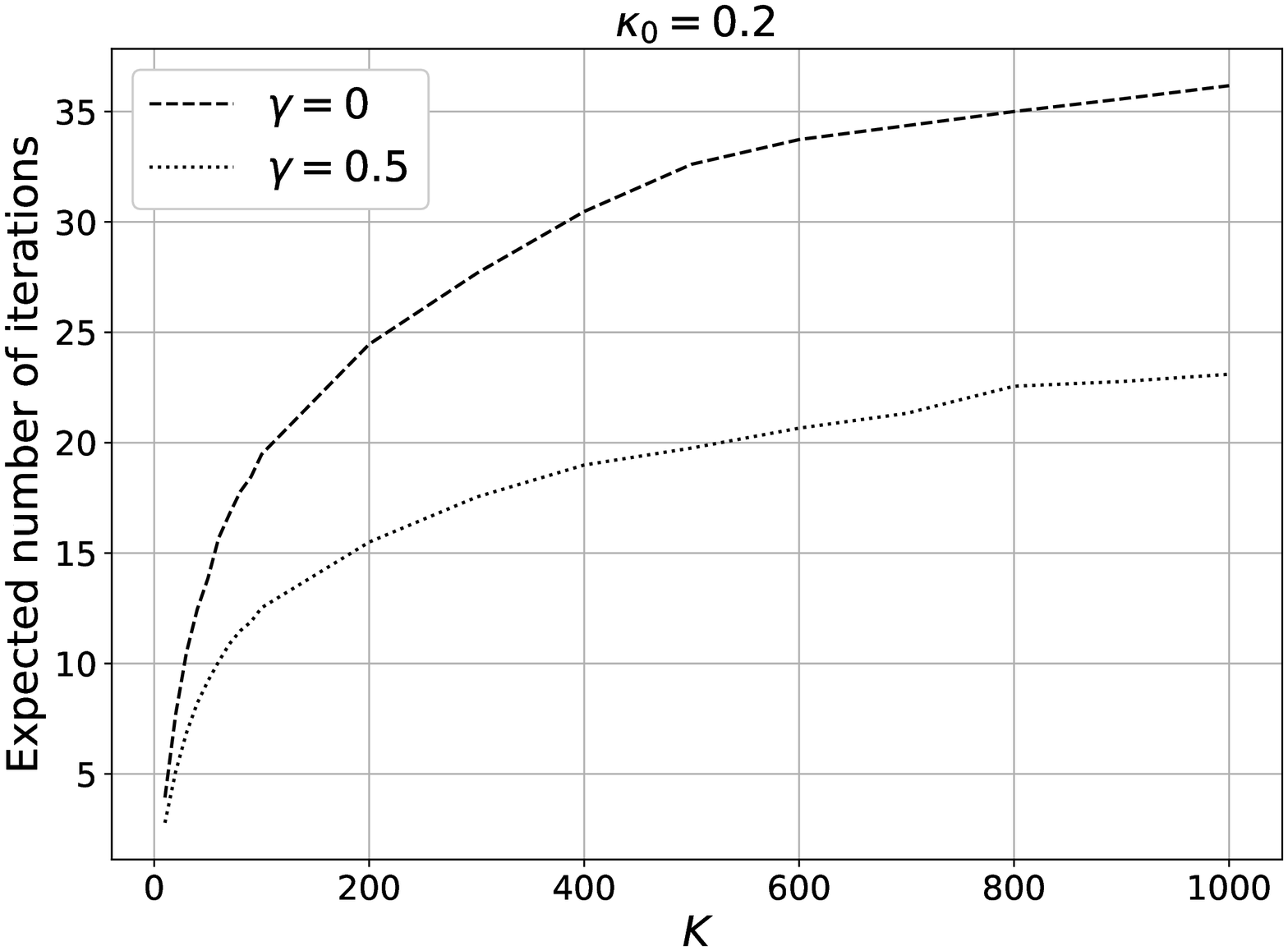}
        \caption{}        
    \end{subfigure}%
  ~
     \begin{subfigure}[t]{0.32\textwidth}
        \centering
        \includegraphics[height=0.8\textwidth]{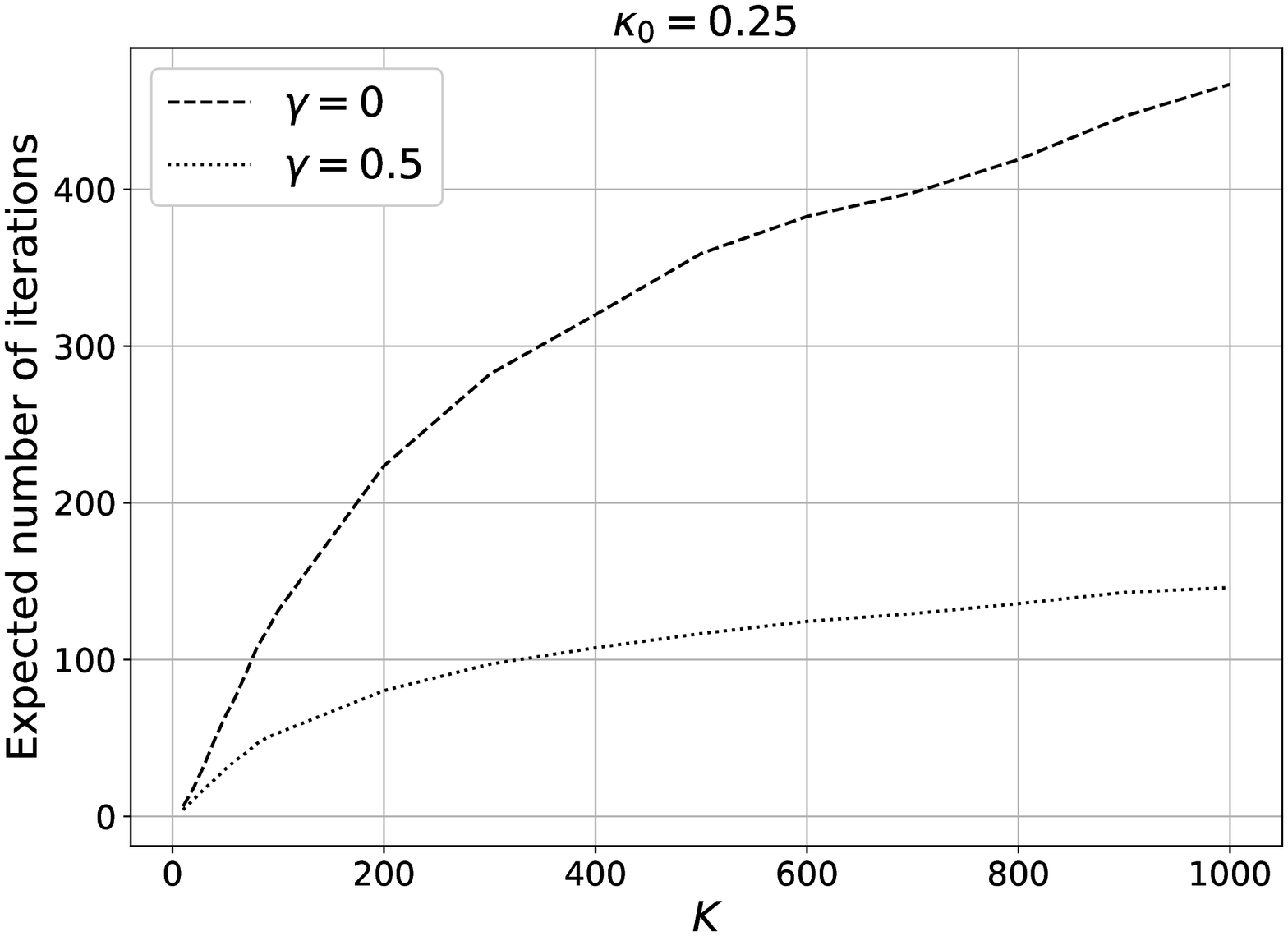}
        \caption{}
    \end{subfigure}%
    \caption{\footnotesize Expected number of iterations of Algorithm~\ref{alg:PRS_GPP} as a function of $K$ for Strauss process with the interaction range $2r = 1/K$.}
    \label{fig:sim_strauss}
\end{figure*}
{Figure~\ref{fig:sim_GJ_vs_Our} corresponds a hard-core process with the interaction range $2r = 0.01$, and it compares the expected number of iterations of the new algorithm with that of the method proposed by \citet{Gj18}, for different values of $\kappa_0$. The complexity of Algorithm~\ref{alg:PRS_GPP} is slightly higher than that of \citet{Gj18}. However, as mentioned earlier, the algorithm of \citet{Gj18} is restricted to hard-core processes, where as the new algorithm can be applied to more general Gibbs processes.}\\
\ \\
\vspace{-5mm}
\begin{figure*}[h!]
    \centering
    \includegraphics[height=0.345\textwidth]{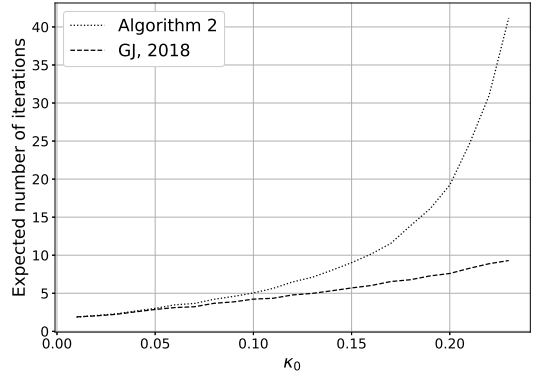}
    \caption{\footnotesize Comparison of expected number of iterations of Algorithm~\ref{alg:PRS_GPP} and the algorithm of \citet{Gj18} for a hard-core model with the interaction range $2r = 0.01$. }
    \label{fig:sim_GJ_vs_Our}
\end{figure*}

\noindent
{\bf PSM model:}
Consider the PSM model with the interaction range $2r = 1/K$ for some $K \geq 1$. We take the intensities of type-$1$ and type-$2$ points to be $\displaystyle \kappa^{(1)}_{0}/(\pi r^2)$ and $\displaystyle \kappa^{(2)}_{0}/(\pi r^2)$, respectively. Figure~\ref{fig:sim_PSM} plots the estimated expected number of iterations of the algorithm as a function of $K$. Perfect samples from $\mu_i$ on each cell $i$ are generated using the dCFTP method by \citet{KW98}. Again, we see that the expected number of iterations of the algorithm seems to be $O\lt(\log(1/r)\rt)$ for small values of $\kappa_0 = \kappa^{(1)}_{0} + \kappa^{(2)}_{0}$.

\begin{figure*}[h!]
    \centering
    \includegraphics[height=0.345\textwidth]{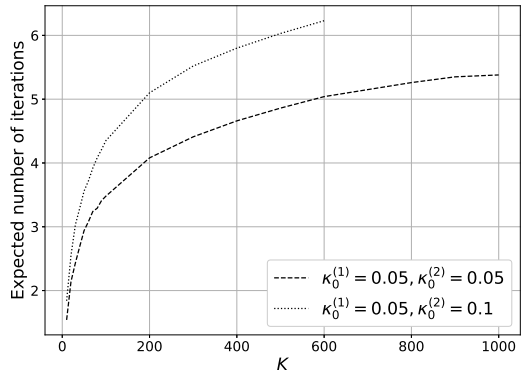}
    \caption{\footnotesize Expected number of iterations of Algorithm~\ref{alg:PRS_GPP} for PSM model.}
    \label{fig:sim_PSM}
\end{figure*}
\section{Conclusion} 
\label{sec:con}
In this paper, we considered the problem of perfect sampling for Gibbs point processes with a finite interaction range $2r$, defined on $S\subseteq \reals^d$. We proposed a new perfect sampling algorithm by combining the existing perfect sampling methods and the partial rejection sampling proposed by \citet{GJL17}. For pairwise interaction processes, penetrable spheres mixture models, and area-interaction processes that are absolutely continuous with respect to a $\kappa$-homogeneous Poisson point process on $S = [0,1]^d$, we showed that if $\kappa = \kappa_0/(\mathsf{v}_d r^d)$, the proposed algorithm can be implemented with the expected running time complexity of $O(\log(1/r))$ as $r$ goes to $0$, for sufficiently small values of $\kappa_0$. We illustrated our findings using several simulation results. From these simulations, we notice that the value of $\kappa_0$ can be at least $0.25$ for Strauss processes. However, at this stage, we do not have a theoretical justification to support this claim, and we would like to address this in future research. 

\section*{Acknowledgements}
This work has been supported by the Australian Research Council Centre of Excellence for Mathematical and Statistical Frontiers (ACEMS), under grant number CE140100049. 
We would like to thank Michel Mandjes for bringing the partial rejection sampling to our attention.

\bibliographystyle{apalike}
\bibliography{Ref}

\end{document}